\documentclass[12pt,a4paper,twoside]{article}
\usepackage{bm, amsmath, amssymb, amsthm} 
\usepackage{graphicx}

\newtheorem{theorem}{Theorem}
\newtheorem{example}{Example}
\newtheorem{lemma}{Lemma}
\newtheorem{proposition}{Proposition}
\newtheorem{remark}{Remark}

\def\eq{\hspace*{-1.5mm}&=&\hspace*{-1.5mm}}
\def\RR{\mathbb{R}}
\def\<{\langle}
\def\>{\rangle}
\def\n{N}
\def\T{T}

\title{The curve shortening flow in the metric-affine plane}
\author{Vladimir Rovenski\thanks{Department of Mathematics, University of Haifa, Mount Carmel, 31905 Haifa, Israel.
\newline
E-mail: \texttt{vrovenski@univ.haifa.ac.il}}}

\begin{document}

\date{}

\maketitle

\begin{abstract}
We investigate for the first time the curve shortening flow in the metric-affine plane
and prove that under simple geometric condition it shrinks a closed convex curve to a ``round point" in finite time.
This generalizes the classical result by M.~Gage and R.S.~Hamilton about convex curves in Euclidean plane.


\end{abstract}

\section{Introduction}

The one-dimensional mean curvature flow is called the \textit{curve shortening flow} (CSF),
because it is the negative $L^2$-gradient flow of the length of the interface,
and it is used in modeling the dynamics of melting solids.
The CSF deals with a family of closed curves $\gamma$ in the plane $\RR^2$
with a Euclidean metric $g=\<\,\cdot\,,\cdot\,\>$ and the Levi-Civita connection $\nabla$,
satisfying the initial value problem (with parabolic partial differential equation)
\begin{eqnarray}\label{E-01}
 {\partial\gamma}/{\partial t}=k\n,\quad \gamma|_{\,t=0}=\gamma_{0}.
\end{eqnarray}
Here, $k$ is the curvature of $\gamma$ with respect the unit inner normal vector $\n$
and $\gamma_{0}$ is an embedded plane curve, see survey in \cite{cz,zhu}.
The flow defined by \eqref{E-01} is invariant under translations and rotations
Recall that the curvature of a convex plane curve is positive.
The next theorem by M.~Gage and R.S.~Hamilton \cite{GH86} describes this flow of convex curves.

\begin{theorem}\label{T-01z}
 $a)$ Under the CSF \eqref{E-01}, a convex closed curve in the Euclidean plane smoothly shrinks to a~point in finite time.
 $b)$ Rescaling in order to keep the length constant, the flow converges exponentially fast to a circle in $C^\infty$.
\end{theorem}

This theorem and further result by M.A.~Grayson, \cite{Gr} (that the flow moves any closed embedded in the Euclidean plane curve in a finite time to a convex curve)
have many generalizations and applications in natural and computer sciences.
For example, the \textit{anisotropic curvature-eikonal flow} (ACEF) for closed convex curves, see \cite[Section~3.4]{cz},
\begin{equation}\label{E-20}
  {\partial\gamma}/{\partial t}= (\Phi(\theta)k + \lambda\,\Psi(\theta))\,\n,\quad
 \gamma|_{\,t=0} =\gamma_{0} ,
\end{equation}
where $\Phi>0$ and $\Psi$ are $2\pi$-periodic functions of the normal to $\gamma(\cdot\,,t)$ angle $\theta$
and $\lambda\in\RR$, generalizes the CSF.
Anisotropy of
\eqref{E-20} is indispensable in dealing with
phase transition, crystal growth, frame propagation, chemical reaction, and mathematical biology.
The~particular case of ACEF, when $\Phi$ and $\Psi$ are positive constants, serves as a model for essential biological processes, see \cite{hww}.
On the other hand, \eqref{E-20} is a particular case of the flow
\begin{equation*}
 {\partial\gamma}/{\partial t}=F(\gamma,\theta,k)\n,\quad \gamma|_{\,t=0}=\gamma_{0},
\end{equation*}
see \cite[Chapter~1]{cz},
where $F=F(x,y,\theta,q)$ is a given function in $\RR^4$, $2\pi$-periodic in $\theta$.

During last decades, many results have appeared  in the differential geometry of a manifold with an affine connection $\bar\nabla$
(which is a method for transporting tangent vectors along curves),
e.g., collective monographs \cite{cl,mikes}. The difference ${\mathfrak T}=\bar\nabla-\nabla$
(of $\bar\nabla$ and the Levi-Civita connection $\nabla$ of $g$), is a (1,2)-tensor, called \textit{contorsion tensor}.
Two~interes\-ting particular cases of $\bar\nabla$ (and $\mathfrak{T}$) are as follows.

1) \textit{Metric compatible connection}: $\bar\nabla g=0$, i.e., $\<\mathfrak{T}(X,Y),Z\>=-\<\mathfrak{T}(X,Z),Y\>$.
Such manifolds appear in almost Hermitian and Finsler geometries and
are central in Einstein-Cartan theory of gravity, where the torsion is represented by the spin tensor of matter.

2) \textit{Statistical connection}: $\bar\nabla$ is torsionless and the rank 3 tensor $\bar\nabla g$ is symmetric
in all its entries, i.e., $\<\mathfrak{T}(X,Y),Z\>$ is fully symmetric.
Statistical manifold structure, which is related to geometry of a pair of dual affine connections,
is central in Information Geometry, see \cite{Hi}; affine hypersurfaces in $\RR^{n+1}$ are a natural source of statistical manifolds.

There are no results about the CSF in metric-affine geometry.
The \textit{metric-affine plane} is $\RR^2$ endowed with
a Euclidean metric $g$ and an affine connection $\bar\nabla$. 
Our objective is to study the CSF in the metric-affine plane and to generalize
Theorem~\ref{T-01z} for convex curves in $(\RR^2,g,\bar\nabla)$.
Thus, we replace \eqref{E-01} by the following initial value problem:
\begin{equation}\label{E-05}
 {\partial\gamma}/{\partial t}= \bar k \,\n,\quad \gamma|_{\,t=0} =\gamma_{0},
\end{equation}
where $\bar k$ is the curvature of a curve $\gamma$ with respect to $\bar\nabla$ and $\gamma_0$ is a closed convex curve.
Note that \eqref{E-05} is the particular case (when $\Phi=1$ and $\lambda\,\Psi(\theta)=\Psi(\theta)$) of the ACEF. 
 Put
\[
 k_0:=\min\{k(x):\,x\in\gamma_0\}>0.
\]
Let $\{e_{1},e_{2}\}$ be the orthonormal frame in $(\RR^{2},g,\bar\nabla)$.
In the paper we assume that
\begin{equation}\label{E-05p}
 \mbox{{the contorsion tensor \ ${\mathfrak T}$ \ is \ $\nabla$-parallel}},
\end{equation}
i.e., ${\mathfrak T}$ has constant components ${\mathfrak T}^k_{ij}=\<{\mathfrak T}(e_i,e_j),\,e_k\>$
and constant norm $\|{\mathfrak T}\|=c\ge0$.

 Let $\gamma: S^{1}\to\RR^{2}$ be a closed curve in the metric-affine plane with the arclength parame\-ter~$s$.
Then $\T=\partial\gamma/\partial s$ is the unit vector tangent to $\gamma$.
In this case, $k=\<\nabla_{\T}\,\T,\n\>$ and
the curvature of $\gamma$ with respect to an affine connection $\bar\nabla$ is $\bar k=\<\bar\nabla_{\T}\,\T,\n\>$,
we obtain
\begin{equation}\label{E-04}
 \bar k = k + \Psi,
\end{equation}
where $\Psi$ is the following function on $\gamma$:
\begin{equation}\label{E-01a}
 \Psi = \<{\mathfrak T}(\T,\T), \n\>.
\end{equation}
By the assumptions $\|{\mathfrak T}\|=c$, see \eqref{E-05p}, and $\|T\|=\|N\|=1$, we have
\begin{equation}\label{E-01b}
 |\,\Psi|\le c.
\end{equation}
The convergence of the ACEF \eqref{E-20} when $\Phi$ and $\Psi$ are positive
has been studied in \cite[Chapter~3]{cz}.
However, our function $\Psi$ in \eqref{E-04} takes both positive and negative values,
and \cite[Theorem~3.23]{cz} is not applicable to our flow of \eqref{E-05}.
By~this reason, we independently develop the geometrical approach to prove the convergence of \eqref{E-05} to a "round point".
Our~main goal is the following theorem, generalizing Theorem~\ref{T-01z}(a).

\begin{theorem}\label{T-01main}
Let $\gamma_{0}$ be a closed convex curve in the metric-affine plane with condition $k_0 > 2\,c$.
Then \eqref{E-05} has a unique solution $\gamma(\cdot,t)$,
and it exists
at a finite time interval $[\,0,\omega)$, and as $t\uparrow\omega $, the solution $\gamma(\cdot,t)$ converges to a point.
Moreover, if $\,k_0 > 3\,c$ then
 $\omega\le \frac{A(\gamma_0)}{2\pi}\cdot\frac{k_0-2c}{k_0-3c}$,
 {where}
 $A(\gamma_0)$
 {is the area enclosed by}
 $\gamma_0$.
\end{theorem}

Nonetheless, the approach of \cite{cz}
to the norma\-lized flow of \eqref{E-20} in the contracting case still works without the positivity of $\Psi$, see \cite[Remark~3.14]{cz}.	
Based on this result and Theorem~\ref{T-01main}, we obtain the following result, generalizing Theorem~\ref{T-01z}(b).

\begin{theorem}\label{T-02main}
Consider the normalized curves $\tilde\gamma(\cdot, t) = (2(\omega-t))^{1/2}\gamma(\cdot, t)$, see Theorem~\ref{T-01main},
and introduce a new time variable
 $\tau = -(1/2)\log (1-\omega^{-1}\,t) \in [\,0,\infty)$.
 Then the curves $\,\tilde\gamma(\cdot, \tau)$ converge to the unit circle smoothly as $\tau\to\infty$.
\end{theorem}

In Section~\ref{sec:T1}, we prove Theorem~\ref{T-01main} in several steps,
some of them generalize the steps in the proof of \cite[Theorem~1.3]{zhu}.
In Section~\ref{sec:T2}, we
prove Theorem~\ref{T-02main} about the normalized flow \eqref{E-05},
following the proof of convergence of the normalized flow \eqref{E-20} in the contracting~case.

Theorem~\ref{T-01main} can be easily extended to the case of non-constant contorsion tensor $\mathfrak{T}$ of small norm, but we can not now reject the assumption \eqref{E-05p} for Theorem~\ref{T-02main}, since its proof is based on the result for the normalized ACEF, see \cite{cz}, where $\Psi$ depends only on $\theta$.

\section{Proof of Theorem~2}
\label{sec:T1}

Recall the axioms of affine connections $\bar\nabla:\mathfrak{X}_M\times\mathfrak{X}_M\to\mathfrak{X}_M$ on a manifold $M$, e.g., \cite{mikes}:
\[
 \bar\nabla_{f X_1+X_2}Y = f\,\bar\nabla_{X_1}Y + \bar\nabla_{X_2}Y,\quad
 \bar\nabla_X(f Y_1+Y_2) = f\,\bar\nabla_XY_1 + X(f)\cdot Y_1 + \bar\nabla_XY_2
\]
for any vector fields $X, Y, X_1, X_2, Y_1, Y_2$ and smooth function $f$ on $M$.

Let $\theta$ be the \textit{normal angle} for a convex closed curve $\gamma: S^{1}\to\RR^{2}$, i.e.,
 $\cos\theta =-\<\n, e_{1}\>$ and $\sin\theta =-\<\n, e_{2}\>$.
Hence,
\begin{equation}\label{E-01c}
 \n=-[\cos\theta,\, \sin\theta],\quad
 \T= [-\sin\theta,\, \cos\theta].
\end{equation}

\begin{lemma}
The function $\Psi$ given in \eqref{E-01a} has the following view in the coordinates:
\begin{equation}\label{E-phi0}
  \Psi = a_{30}\sin^3\theta + a_{03}\cos^3\theta +a_{12}\sin \theta + a_{21}\cos \theta,
\end{equation}
where $a_{ij}$ are given by
\begin{eqnarray}\label{E-phi}
\nonumber
 && a_{12}={\mathfrak T}^{2}_{12}+{\mathfrak T}^{2}_{21}-{\mathfrak T}^{1}_{11},\qquad\quad\ \
 a_{21}={\mathfrak T}^{1}_{12}+{\mathfrak T}^{1}_{21}-{\mathfrak T}^{2}_{22},\\
 && a_{03} = {\mathfrak T}^{2}_{22} -{\mathfrak T}^{1}_{22} -{\mathfrak T}^{1}_{12} -{\mathfrak T}^{1}_{21},\quad
 a_{30} = {\mathfrak T}^{1}_{11} -{\mathfrak T}^{2}_{11} - {\mathfrak T}^{2}_{12}-{\mathfrak T}^{2}_{21}.
\end{eqnarray}
\end{lemma}

\begin{proof}
Using \eqref{E-01c}, we find
\begin{eqnarray*}
 {\mathfrak T}(\T,\T) \eq {\mathfrak T}(e_1,e_1)\sin^2\theta - ({\mathfrak T}(e_1,e_2)+{\mathfrak T}(e_2,e_1))\sin\theta\cos\theta
  +{\mathfrak T}(e_2,e_2)\cos^2\theta ,\\
 \<{\mathfrak T}(\T,\T),\n\> \eq -{\mathfrak T}^{2}_{11}\sin^{3}\theta
 +({\mathfrak T}^{2}_{12}{+}{\mathfrak T}^{2}_{21}{-}{\mathfrak T}^{1}_{11})\sin^{2}\theta\cos\theta \\
 &&\hskip-1.5mm +\,({\mathfrak T}^{1}_{12}{+}{\mathfrak T}^{1}_{21}{-}{\mathfrak T}^{2}_{22})\sin\theta\cos^{2}\theta
 -{\mathfrak T}^{1}_{22}\cos^{3}\theta.
\end{eqnarray*}
From this and the definition ${\mathfrak T}_{ij}=\sum_{\,k}{\mathfrak T}^{k}_{ij}\,e_{k}$
the equalities \eqref{E-phi0} and \eqref{E-phi} follow.
\end{proof}

\begin{remark}\rm
By equalities $\T_\theta=\n$, $\n_\theta=-\T$ and \eqref{E-01a}, we obtain the following:
\begin{equation*}
 |\Psi_{\theta\theta}+\Psi\,|
 = |2\<{\mathfrak T}(\n,\n), \n\> -4\<{\mathfrak T}(\n,\T), \T\>-3\<{\mathfrak T}(\T,\T), \n\>|
 \le 9\,c.
\end{equation*}
\end{remark}

\begin{example}\rm
Recall the Frenet--Serret formulas (with the $\nabla$-curvature $k$ of $\gamma$):
\begin{equation}\label{E-02}
 \nabla_{\T}\,\T = k\, \n,\quad
 \nabla_{\T}\,\n = -k\,\T.
\end{equation}
 For the affine connection $\bar\nabla$, using \eqref{E-02} we obtain
\begin{equation}\label{E-03}
 \bar\nabla_{\T}\,\T =k\,\n+{\mathfrak T}(\T,\T),\quad
 \bar\nabla_{\T}\,\n =-k\,\T+{\mathfrak T}(\T,\n).
\end{equation}
By \eqref{E-03}, the Frenet--Serret formulas $\bar\nabla_{\T}\,T=\bar k\,N$ and $\bar\nabla_{\T}\,N=-\bar k\,T$
(with the $\bar\nabla$-curvature $\bar k$ of $\gamma$) hold for any curve $\gamma$
if and only if
\[
 \<{\mathfrak T}(\T,\T), \n\>=-\<{\mathfrak T}(\T,\n), \T\>,\quad
 \<{\mathfrak T}(\T,\n), \n\>=0=\<{\mathfrak T}(\T,\T), \T\>.
\]
In this case, we have in coordinates the following symmetries:
\begin{eqnarray*}
 && {\mathfrak T}^{1}_{12}=-{\mathfrak T}^{2}_{11},\quad
 {\mathfrak T}^{1}_{21}=0={\mathfrak T}^{2}_{22},\quad
 {\mathfrak T}^{2}_{21}=-{\mathfrak T}^{1}_{22},\quad
 {\mathfrak T}^{2}_{12}=0={\mathfrak T}^{1}_{11},\\
 && a_{12}=-{\mathfrak T}^{1}_{22},\quad
 a_{21}=-{\mathfrak T}^{2}_{11},\quad
 a_{03} = -a_{30} = {\mathfrak T}^{2}_{11} -{\mathfrak T}^{1}_{22},
\end{eqnarray*}
and the formula
 $\,\Psi =
 a_{30}(\sin 3\theta -\cos 3\theta)
 +a_{12}\sin\theta
 +a_{21}\cos\theta$.
\end{example}

The~\textit{support function} $S$ of a convex curve $\gamma$ is given by, e.g. \cite{zhu},
\begin{equation}\label{E-S1}
 S(\theta)=\<\gamma(\theta),\, -\n\> = \gamma^1(\theta)\cos\theta+\gamma^2(\theta)\sin\theta.
\end{equation}
For example, a circle of radius $\rho$ has $S(\theta)\equiv\rho$.
Since $\<\,\partial\gamma/\partial\theta, \n\>=0$, the derivative $S_{\theta}$ is
\begin{eqnarray*}
  S_{\theta}(\theta) = -\gamma^1(\theta)\sin\theta+\gamma^2(\theta)\cos\theta,
\end{eqnarray*}
and $\gamma$ can be represented by the support function and parameterized by $\theta$, see \cite{zhu},
\begin{equation}\label{E-06}
 \gamma^{1}= S\cos\theta- S_\theta\sin\theta,\quad
 \gamma^2=S\sin\theta+S_\theta\cos\theta.
\end{equation}
This yields the following known formula for the curvature of $\gamma(\theta)$:
\begin{equation}\label{E-07}
 k=({S_{\theta\theta}+S})^{-1}.
\end{equation}
Then, according to \eqref{E-04} and \eqref{E-07},
\begin{equation}\label{E-08}
 \bar k = ({S_{\theta\theta}+S})^{-1} + \Psi.
\end{equation}

Let $\widehat\gamma(u,t):S^1\times[\,0,T)\rightarrow \RR^{2}$ be a family of closed curves satisfy\-ing~\eqref{E-05}.
We will use the normal angle $\theta$ to parameterize each curve:
$\gamma(\theta,t)=\widehat\gamma(u(\theta,t),t)$.

\begin{proposition}
The support function ${S}(\cdot\,,t) =\<\gamma(\cdot\,,t), -\n\>$ of $\gamma(\cdot\,,t)$ satisfies the following partial differential equation:
\begin{equation}\label{E-10}
 {S}_t = -({S_{\theta\theta}+S})^{-1} -\Psi.
\end{equation}
\end{proposition}

\begin{proof}
Observe that $\partial\widehat\gamma/\partial u$ is orthogonal to $\n$ and
\begin{equation*}
 \frac{\partial\gamma}{\partial t}
 =\frac{\partial\widehat\gamma}{\partial u}\cdot\frac{{\partial u}}{\partial t}
  +\frac{\partial\gamma}{\partial t}
 =\frac{\partial\widehat\gamma}{\partial u}\cdot\frac{{\partial u}}{\partial t}
  +\bar k\cdot\n.
\end{equation*}
Using this, \eqref{E-S1} and equality $\n_t=0$, see \eqref{E-01c}, we obtain
\begin{equation}\label{E-10a}
 S_t = \frac{\partial}{\partial t} \<\gamma(\theta,t),-\n\>
 =\<\frac{\partial\gamma}{\partial t},-\n\>
 = -\bar k .
\end{equation}
Then we apply \eqref{E-08}.
\end{proof}

By the theory of parabolic equations we have the following.

\begin{proposition}[Local existence and uniqueness]\label{P-loc-exist}
Let $\gamma_0$ be a convex closed curve in the metric-affine plane. Then there exists a unique family of convex closed curves
$\gamma(\cdot,t),\ t\in[\,0,t_0)$ with $t_0>0$, and $\gamma(\cdot,0)=\gamma_0$ satisfying~\eqref{E-05}.
\end{proposition}

\begin{proof}
We will show that \eqref{E-10} is parabolic on $S(\theta,t)$.
To approximate \eqref{E-10} linearly, consider the second order partial differential equation $\partial_{t}{S}=f$ for ${S}(\theta,t)$,
where
\[
 f({S},{S}_{\theta\theta},\theta) = -({{S}_{\theta\theta} +{S}})^{-1}-\Psi.
\]
Take the initial point $\widetilde{S}=(\widehat{S},\widehat{S}_{\theta\theta},\widetilde{\theta})$
and set $h={S} - \widehat{S}$ for the difference of support functions.
Then
\begin{eqnarray*}
 && f({S},{S}_{\theta\theta},\theta)\approx f(\widehat{S},\widehat{S}_{\theta\theta},\widetilde{\theta})
 +\frac{\partial f}{\partial{S}}|_{\widetilde{S}}\cdot h
 +\frac{\partial f}{\partial {S}_{\theta\theta}}|_{\widetilde{S}}\cdot h_{\theta\theta} +\frac{\partial f}{\partial \theta}|_{\widetilde{S}}\cdot(\theta-\tilde\theta),
\end{eqnarray*}
where
 $\frac{\partial f}{\partial {S}}|_{\widetilde{S}} = (\widehat{S}_{\theta\theta}+\widehat{S})^{-2}$,
 $\frac{\partial f}{\partial {S}_{\theta\theta}}|_{\widetilde{S}} = (\widehat{S}_{\theta\theta}+\widehat{S})^{-2}$
 and
 $\frac{\partial f}{\partial \theta}|_{\widetilde{S}} = -\Psi_\theta$.
Hence, the linearized partial differential equation for $h$ is
\begin{equation}\label{E-lnearized}
 \partial_{t}h=(\widehat{S}_{\theta\theta}+\widehat{S})^{-2}\,(h_{\theta\theta}+h)-\Psi_{\theta}\cdot(\theta-\tilde\theta).
\end{equation}
The coefficient $(\widehat{S}_{\theta\theta}+\widehat{S})^{-2}$ of $h_{\theta\theta}$ is positive, therefore, \eqref{E-lnearized} is parabolic.
\end{proof}

\begin{proposition}[Containment principle]\label{L-06}
Let convex closed  curves $\gamma_{1}$ and $\gamma_{2}:\,S^{1}\times[\,0,t_0)\rightarrow \RR^{2}$
in the metric-affine plane be solutions of \eqref{E-05} and $\gamma_2(\cdot,0)$ lie in the domain enclosed by $\gamma_1(\cdot,0)$. Then $\gamma_{2}(\cdot,t)$ lies in the domain enclosed by $\gamma_{1}(\cdot,t)$
for all $t\in [\,0,t_0)$.
\end{proposition}

\begin{proof}
Let ${S}_{i}(\theta,t)$ be the support function of $\gamma_{i}(\cdot,t)$ for $0\le t<t_0$ and $i=1,2$.
These $\gamma_i$ satisfy \eqref{E-05} with the same function $\Psi$. Denote $\widetilde{S}= {S}_{1}- {S}_{2}$.
Since $\gamma_{1}$ and $\gamma_{2}$ are convex for all $t$, their curvatures ${k}_{i}$ are positive.
Using \eqref{E-07} and \eqref{E-10}, we get the parabolic equation
\begin{equation*}
 {\widetilde{S}}_t = k_1 k_2(\widetilde{S}_{\theta\theta}+\widetilde{S})
\end{equation*}
with the initial value $\widetilde S(\theta,0)\ge0$.
Applying the scalar maximum principle of parabolic equations, e.g. \cite[Section~1.2]{cz},
we deduce that $\widetilde S(\theta,t)\ge 0$.
Hence, $\gamma_{2}(\cdot,t)$ lies in the domain enclosed by $\gamma_{1}(\cdot,t)$
for all $t\in [\,0,t_0)$.
\end{proof}

\begin{proposition}[Preserving convexity]\label{L-07}
Let $[\,0,\omega)$ be the maximal time interval for the solution $\gamma(\cdot,t)$ of \eqref{E-05} in the metric-affine plane, and let the curvature of $\gamma_{0}$ obey condition $k_0 > 2\,c$.
Then the solution $\gamma(\cdot,t)$ remains convex on $[\,0,\omega)$ and
its curvature has a uniform positive lower bound $k_0-2c$ for all $t\in[\,0,\omega)$.
\end{proposition}

\begin{proof}
By Proposition~\ref{P-loc-exist}, $\gamma(\cdot,t)$ is convex (i.e., $k>0$) on a time interval $[\,0,\tilde\omega)$ for some $\tilde\omega\le\omega$, and its support function satisfies \eqref{E-10} for
 $(\theta,t)\in S^1\times[\,0,\tilde\omega)$.
Taking derivative of $\bar k$ in $t$, see \eqref{E-08}, we get:
\begin{equation*}
 \bar{k}_{t} = \big(({{S}_{\theta\theta}+{S}})^{-1}\big)_{t}
 = -({S}_{\theta\theta}+{S})^{-2}({S}_{\,t\,\theta\theta} + {S}_{t})
 = k^{2}(\bar{k}_{\theta\theta} + \bar{k}) .
\end{equation*}
Thus, $\bar k(\theta,t)$ satisfies the following parabolic equation:
\begin{equation}\label{E-12}
 \bar{k}_{t} = k^{2}(\bar{k}_{\theta\theta} + \bar{k}) .
\end{equation}
Applying the maximum principle to \eqref{E-12}, we find
 $\min_{\,\theta\in S^1}\bar{k}(\theta,t)\ge\min_{\,\theta\in S^1} \bar{k}(\theta,0) = \bar{k}_0$
for $t \in [\,0,\tilde{\omega})$.
By conditions and \eqref{E-01b},
\[
 \bar k = {k} + \Psi \ge {k} -|\Psi|\ge {k}_0 - c > 0.
\]
This and equality \eqref{E-04} imply that the curvature $k$ of $\gamma(\cdot,t)$ has a uniform positive lower bound $k_0-2\,c$ for all $t\in[\,0,\omega)$.
\end{proof}

\begin{lemma}\label{P-04b}
Let $\gamma_t$ be a solution of \eqref{E-05} in the metric-affine plane with $\Psi$ given in \eqref{E-01a}.
Then in the coordinates, $\tilde\gamma_t=\gamma_t + t[a_{21},a_{12}]$ is a solution of \eqref{E-05} with
the $\bar\nabla$-curvature $\bar k= k+\tilde\Psi$ and $\tilde\Psi = a_{30}\sin^3\theta + a_{03}\cos^3\theta$.
\end{lemma}

\begin{proof}
By \eqref{E-S1}, the support function $\widetilde S(t,\cdot)$ of
the curve $\tilde\gamma_t$, obtained by parallel translation from the curve $\gamma_t$, thus, having the same curvature
$\tilde k=k$, satisfies
\[
 \widetilde S(\theta,t) = S(\theta,t) + t(a_{21}\cos\theta + a_{12}\sin\theta).
\]
This, \eqref{E-04} and \eqref{E-10a} yield $\widetilde S_t= - k - \tilde\Psi$,
where
$\tilde\Psi=\<{\mathfrak T}(\tilde\T,\tilde\T), \tilde\n\>$ is defined for $\tilde\gamma$ and
has the view $\tilde\Psi = \Psi - a_{12}\sin\theta - a_{21}\cos\theta$.
Using \eqref{E-phi0} for $\Psi$, completes the proof.
\end{proof}

By Lemma~\ref{P-04b}, we can assume the equalities $a_{21}=a_{12}=0$, i.e.,
\begin{equation*}
 {\mathfrak T}^{2}_{12}+{\mathfrak T}^{2}_{21}-{\mathfrak T}^{1}_{11}=0,\quad
 {\mathfrak T}^{1}_{12}+{\mathfrak T}^{1}_{21}-{\mathfrak T}^{2}_{22}=0.
\end{equation*}
In abbreviated notation, we will omit `tilde' for $\tilde\Psi$, $\tilde\gamma_t$ and $\widetilde S$.
Hence,
\begin{equation}\label{E-phi-2c}
 \Psi = a_{30}\sin^3\theta + a_{03}\cos^3\theta,\ \
 {\rm where}\ \ a_{03}= {\mathfrak T}^{2}_{22},\ \ a_{30}={\mathfrak T}^{1}_{11}.
\end{equation}
From Lemma~\ref{P-04b}, see also \eqref{E-phi-2c}, we conclude the following.

\begin{proposition}\label{P-reduce}
If $a_{30}=a_{03}=0$, see \eqref{E-phi0} and \eqref{E-phi}, then the problem \eqref{E-05} in the metric-affine plane reduces to the classical problem \eqref{E-01} in the Euclidean plane for modified by parallel translation of $\gamma_t$ curves $\tilde\gamma_t=\gamma_t + t[a_{21},a_{12}]$.
\end{proposition}

\begin{example}\label{L-08}\rm
One may show that
\[
 {S}(\theta,t) = \rho(t) -\epsilon_{1}(t)\sin\theta -\epsilon_{2}(t)\cos\theta
\]
with
\begin{equation}\label{E-circ}
  \rho(t)= \sqrt{\rho^{2}(0)-2t},\quad \epsilon_{1}(t) = a_{12}\,t,\quad
  \epsilon_{2}(t) = a_{21}\,t,\quad 0 \le t \le \rho^{2}(0)/2,
\end{equation}
is the support function of a special solution
of \eqref{E-05} with $a_{30}=a_{03}=0$.
We claim that the solution is a family of round circles of radius $\rho(t)\ (t\ge0)$ shrinking to a point
at the time $t_0=\frac12\,\rho^{2}(0)$.
Indeed, by \eqref{E-06}, $S(\theta,t)$ corresponds to a family of circles
\[
 \gamma_t=[\,\rho(t)\cos\theta-\epsilon_{2}(t),\,\rho(t)\sin\theta-\epsilon_{1}(t)\,]
\]
with centers $(\,-\epsilon_{2}(t),\,-\epsilon_{1}(t)\,)$ and the curvature $k = -{1}/{\rho(t)}$.
We then calculate
\begin{equation*}
 {S}_t = \rho'(t) - \epsilon'_{1}(t)\sin\theta - \epsilon'_{2}(t)\cos\theta, \quad
 {S}_{\theta\theta}  = \epsilon_{1}(t)\sin\theta + \epsilon_{2}(t)\cos\theta.
\end{equation*}
Thus, ${{S}_{\theta\theta}+{S}}= \rho(t)$ holds, and \eqref{E-10} reduces to
\[
 \rho' - \epsilon'_{1}\sin\theta - \epsilon'_{2}\cos\theta = -{1}/{\rho} - \Psi,
\]
where $\theta$ is arbitrary.
We get the system of three ODEs:
\begin{equation*}
 \rho' = -{1}/{\rho},\quad \epsilon'_{1} = a_{12},\quad \epsilon'_{2} = a_{21}.
\end{equation*}
Its solution with initial conditions $\epsilon_{i}(0)=0$ is \eqref{E-circ}.
\end{example}



\begin{example}\label{Ex-semi-symm}\rm
(a)~The \textit{projective connections} $\bar\nabla=\nabla+\mathfrak{T}$ are defined by the condition
\[
 \mathfrak{T}_X Y=\<U,Y\>X+\<U,X\>Y,
\]
where $U$ is a given vector field, e.g., \cite{mikes}.
Then $\Psi=\<\mathfrak{T}_T T,N\>=0$, see \eqref{E-01a}.
Thus, \eqref{E-05} in the metric-affine plane with a projective connection
is equal to \eqref{E-01} in the Euclidean~plane.

(b)~The \textit{semi-symmetric connections} $\bar\nabla=\nabla+\mathfrak{T}$ are defined by the condition
\[
 \mathfrak{T}_X Y=\<U,Y\>\,X-\<X,Y\>\,U,
\]
where $U$ is a given vector field, e.g. \cite{Yano}.
Such connections are metric compatible, and for them the formulas \eqref{E-03} are valid.
The definition \eqref{E-01a} reads 
\[
 \Psi=-\<U,N\> = -\<U,e_1\>\cos\theta -\<U,e_2\>\sin\theta.
\] 
Then, see \eqref{E-phi},
 $a_{30}=\<U, e_2-e_1\>=-a_{03}$.
Let $U$ be a constant vector field on $\RR^2$, then we can take the orthonormal frame $\{e_{1},e_{2}\}$ in $(\RR^{2},g,\bar\nabla)$
such that $U$ is orthogonal to $e_{1}-e_{2}$.
Thus, see Proposition~\ref{P-reduce}, the problem \eqref{E-05} in the metric-affine plane with a semi-symmetric connection
and constant $U$ reduces to the problem \eqref{E-01} in the Euclidean~plane.
\end{example}

\begin{proposition}[Finite time existence]
Let a convex closed curve $\gamma_0$ in the metric-affine plane with condition $k_0 > 2\,c$
be evolved by \eqref{E-05}.
Then, the solution $\gamma_{t}$ must be singular at some time $\omega>0$.
\end{proposition}

\begin{proof}
By Lemma~\ref{P-04b} and Example~\ref{L-08},
using translations we can assume $\Psi=a_{30}\sin^3\theta + a_{03}\cos^3\theta$,
see \eqref{E-phi-2c}. Then we calculate
\begin{equation*}
 a_{30}\sin^3\theta + a_{03}\cos^3\theta
 = -\frac14\sqrt{a_{30}^2 + a_{03}^2}\,\big(\underline{\cos(\theta-\theta_0)}-\cos(3\theta+\theta_0)\big) \underline{+\,a_{30}\sin\theta  + a_{03}\cos\theta}
\end{equation*}
for some $\theta_0$.
By Lemma~\ref{P-04b} again and using the rotation $\theta\to\theta-\theta_0$, the underlined terms can be canceled, and the retained expression will be $\frac14\sqrt{a_{30}^2 + a_{03}^2}\cos(3\theta+\theta_0)$, which can be
reduced to simpler form $\tilde a\sin^3\theta$ for some $\tilde a\in\RR$, using the identity $\sin 3\theta =3\sin\theta-4\sin^3\theta$.

 Thus, we may assume $\Psi=\tilde a\sin^3\theta$ with $\tilde a<0$.
Let
$\gamma_{0}$ lies in a circle $\Gamma_{0}$ of radius
\[
 \rho(0)\ge \max_{\,\theta\in S^1}S(\cdot,0)/(1-2c/k_0)
\]
and centered at the origin $O$.
Let evolve $\Gamma_{0}$ by \eqref{E-05} to obtain a solution~$\Gamma(\cdot,t)$ with support function $S^{\,\Gamma}$.
By~Proposition~\ref{L-06}, $\gamma_{t}$ lies in the domain enclosed by $\Gamma(\cdot,t)$, thus, $S\le S^{\,\Gamma}$.
 Consider two families of circles, see Example~\ref{L-08},
\[
 \Gamma^\pm_t = [\rho(t)\cos\theta,\,\rho(t)\sin\theta \pm t \tilde a],\quad
 \rho(t)= \sqrt{\rho^{2}(0)-2\,t},
\]
being solutions of \eqref{E-05}, hence, having support functions satisfying \eqref{E-10},
\begin{eqnarray*}
 S^\pm_t = (S^\pm_{\theta\theta}+S^\pm)^{-1} \mp t\tilde a\sin\theta =\rho(t)\mp t\tilde a\sin\theta.
\end{eqnarray*}
By Proposition~\ref{L-06}, $S_t \le S^{\,\Gamma}_t$ holds,
and since $|\sin^3\theta|\le|\sin\theta|$, we also have
\begin{eqnarray*}
 S^{\,\Gamma}_t \le \Big\{\begin{array}{cc}
              S^+_t, & 0\le\theta\le\pi, \\
              S^-_t, & \pi\le\theta\le2\pi.
            \end{array}
\end{eqnarray*}
Hence, $\Gamma_t$ lies (in $\RR^2$) below any tangent line to the upper semicircle $\Gamma^+_t$
and above any tangent line to the lower semicircle $\Gamma^-_t$.
Thus, $\Gamma_t\subset{\rm conv}(\Gamma^+_t\cup\Gamma^-_t)$.
The solution $\Gamma^\pm(\cdot,t)$ exists only at a finite time interval $[\,0,\tau]$ with $\tau=\rho^{2}(0)/2$, and $\Gamma^\pm(\cdot,t)$ converges,  as $t\rightarrow\tau$, to a point $\Gamma^\pm_\tau=[\,0,\pm\tilde a\tau]$.
Hence, the convex hull of $\Gamma^+_{t}\cup \Gamma^-_{t}$ shrinks to the line segment with the endpoints
$(0,\tilde a\tau)$ and $(0,-\tilde a\tau)$. We conclude that the solution $\gamma_{t}$ must be singular at some time $\omega\le\tau$.
\end{proof}

\noindent
Note that a point or a line segment are the only compact convex sets of zero area in~$\RR^2$.

\begin{lemma}[Enclosed area]
Let a convex closed curve $\gamma_0$ in the metric-affine plane with condition $k_0 > 2\,c$
be evolved by \eqref{E-05}. Then
$\gamma(\cdot,\omega)$ is either a point or a line segment.
\end{lemma}

\begin{proof}
Suppose the lemma is not true. We may assume the origin is contained in the interior of the region enclosed by $\gamma(\cdot,\omega)$.
We can draw a small circle, with radius $2\rho$ and centered at the origin, in the interior of the region enclosed by $\gamma(\cdot,\omega)$.

Since the solution $\gamma(\cdot,\omega)$ becomes singular at the time $\omega$, we know from the evolution equation \eqref{E-07} that the curvature $k(\cdot,t)$ becomes unbounded as $t\rightarrow \omega$.
To derive a contradiction, we only need to get a uniform bound for the curvature.
Consider
\[
 \phi=\frac{-{S}_{t}}{{S}-\rho}\overset{\eqref{E-10}}=\frac{\bar{k}}{{S}-\rho}.
\]
For any $\tilde\omega < \omega$, we can choose $(\theta_{0},t_{0})$ such that
\[
 \phi(\theta_{0},t_{0})= \max\{\,\phi(\theta,t):\, (\theta,t)\in S^{1}\times [\,0,\tilde\omega]\,\}.
\]
Without loss of generality, we may assume $t_{0}>0$.
Then at $(\theta_{0},t_{0})$,
\begin{eqnarray}\label{E-22}
\nonumber
 0 = \phi_{\theta} = \frac{-{S}_{t\theta}}{{S}-\rho} + \frac{{S}_{t}{S}_{\theta}}{({S}-\rho)^{2}},\\
 \nonumber
 0 \le \phi_{t} = \frac{-{S}_{tt}}{{S}-\rho} + \frac{{S}_{t}^{2}}{({S}-\rho)^{2}},\\
 0\ge \phi_{\theta\theta}= -\frac{{S}_{\theta\theta t}}{{S}-\rho} + \frac{{S}_{t}{S}_{\theta\theta}}{({S}-\rho)^{2}}.
\end{eqnarray}
On the other hand,
\begin{eqnarray*}
 && {S}_{tt}= -\bar{k}_{t}= -k^{2}(\bar k_{\theta\theta}+\bar k) ,\\
 && 0\overset{(\ref{E-22}b)}\le \phi_{t}
 =\frac{ k^{2}(\bar k_{\theta\theta}+ \bar k)}{{S}-\rho} +\frac{\bar{k}^{2}}{({S}-\rho)^{2}} \\
 && = \frac{k^{2}((-{S}_{t})_{\theta\theta}+\bar k)}{{S}-\rho} +\frac{\bar{k}^{2}}{({S}-\rho)^{2}}
 =\frac{k^{2}}{{S}-\rho}\,(-{S}_{t\theta\theta})
 +\frac{k^{2}\bar k
 }{{S}-\rho} +\frac{\bar{k}^{2}}{({S}-\rho)^{2}}.
\end{eqnarray*}
By the above,
\begin{eqnarray*}
 && 0\le \phi_{t}\overset{(\ref{E-22}c)}\le k^2\frac{-S_{t}S_{\theta\theta}}{({S}-\rho)^2} +\frac{k^{2}\bar k}{{S} -\rho} +\frac{\bar{k}^{2}}{({S}-\rho)^{2}} \\
 && =\frac{k^{2}\bar k}{({S}-\rho)^{2}}\big({S}_{\theta\theta}+{S}-\rho\big) +\frac{\bar{k}^{2}}{({S}-\rho)^{2}}
 \overset{(\ref{E-07})}=\frac{\bar k(k+\bar k -\rho k^2)}{({S}-\rho)^{2}}.
\end{eqnarray*}
Since $\bar k = k + \Psi(\theta)$ with $\bar k>0$,
see the proof of Proposition~\ref{L-07}, and using \eqref{E-01b}, we obtain
\begin{equation}\label{E-16a}
 \rho k^{2}\le k+\bar k = 2k +\Psi(\theta) \le 2k +c.
\end{equation}
From quadratic inequality \eqref{E-16a} we conclude that
\[
 0\le k\le \big(1+\sqrt{1+c\rho}\big)\rho^{-1}<\infty \quad {\rm on}\ \  S^1\times[\,0,\omega).
\]
Thus, $k$ is bounded as $t\uparrow\omega$, -- a contradiction.
Thus, the area enclosed by $\gamma(\cdot,t)$ tends to zero as $t\uparrow\omega$.
\end{proof}

The area enclosed by the convex curve $\gamma(\cdot,t)\subset\RR^2$, e.g., \cite[p.~6]{zhu}, is calculated by
\begin{equation}\label{E-area}
 A(t)= -\frac12\int_{\gamma(\cdot,t)}\<\gamma(\cdot,t),\n\>\,ds
 =\frac1{2}\int_0^{2\pi}\frac{S}{k}\,d\theta .
\end{equation}

\begin{proposition}
Let a convex closed curve $\gamma_0$ in the metric-affine plane be evolved by \eqref{E-05}.
If $k_0>3c$ then the maximal time $\omega$ is estimated by
\begin{equation}\label{E-18}
 \omega\le \frac{A(0)}{2\pi}\cdot\frac{k_0-2c}{k_0-3c}.
\end{equation}
\end{proposition}

\begin{proof}
Using \eqref{E-10a}, \eqref{E-12} and the identity
$\int_0^{2\pi} S(\bar k_{\theta\theta}+\bar k)\,d\theta = \int_0^{2\pi}(S_{\theta\theta}+S)\,\bar k\,d\theta$, we get
\begin{eqnarray*}
 \frac{d}{dt}\,A(t) = \frac1{2}\int_0^{2\pi}\frac{S_t k -S k_t}{k^2}\,d\theta
 = -\frac1{2}\int_0^{2\pi}\big[1+\Psi/k +S(\bar k_{\theta\theta}+\bar k)\,\big]\,d\theta  \\
 = -\frac1{2}\int_0^{2\pi}\big[1+\Psi/k +(S_{\theta\theta}+S)\bar k\,\big]\,d\theta
 = -2\pi - \int_0^{2\pi}\frac{\Psi(\theta)}{k(\theta,t)}\,d\theta.
\end{eqnarray*}
Using the inequality $k(\theta,t) \ge k_0 - 2c$, see Lemma~\ref{L-07}, and
$|\,\Psi|\le c$, see \eqref{E-01b}, we get
\begin{equation*}
 \frac{d}{dt}\,A(t) \le -2\pi+\frac{2\,\pi c}{k_0-2c}.
\end{equation*}
By this, we have $A(0)\ge
2\pi\frac{k_0-3c}{k_0-2c}\,\omega$.
Hence, the inequality \eqref{E-18} holds when $k_0>3c$.
\end{proof}

Question: can one estimate $\omega$ when $2\,c<k_0\le 3\,c$\,?

\smallskip

To complete the proof of Theorem~\ref{T-01main}, observe that
if the flow \eqref{E-05} does not converge to a point as the enclosed by $\gamma(\cdot,t)$ area tends to zero, then $\min\limits_{\theta\in S^1} k(\theta,t)$ tends to zero as $t\uparrow\omega$, -- a contradiction to Proposition~\ref{L-07}.

\section{Proof of Theorem~3}
\label{sec:T2}

Here, we study the normalized flow \eqref{E-05}.
From \eqref{E-area} we have
\[
 \frac{A(t)}{2(\omega-t)} = \pi +\frac{1}{2(\omega-t)}\int_t^\omega\int_\gamma \Psi(\theta)\,ds\,dt,
\]
hence,
 $\lim_{\,t\uparrow\omega}\frac{A(t)}{2(\omega-t)} = \pi$.
Without loss of generality, we may assume that the flow shrinks at the origin.
Thus, we rescale the solution $\gamma(\cdot, t)$ of \eqref{E-05} as
\[
 \tilde\gamma(\cdot, t) = (2(\omega-t))^{1/2}\gamma(\cdot, t).
\]
The corresponding support function and curvature are given by
\[
 \widetilde S(\cdot,t) = (2(\omega-t))^{1/2} S(\cdot,t),\quad
 \tilde k(\cdot,t) = (2(\omega-t))^{-1/2} k(\cdot,t).
\]
Introduce a new time variable $\tau\in [\,0,\infty)$ by
 $\tau = -(1/2)\log (1-\omega^{-1}\,t)$.
Using the above definitions, we find the partial differential equation for $\widetilde S$,
\begin{equation}\label{E-18a}
 \widetilde S_{\tau} = -(\tilde k +\sqrt{\,2\,\omega}\,e^{-\tau}\Psi) + \widetilde S,
\end{equation}
and that the normalized curvature, i.e., of the curves $\tilde\gamma(\cdot, \tau)$, satisfies the equation
\begin{equation*}
 \tilde k_{\tau} = \tilde k^2(\tilde k_{\theta\theta}+\tilde k) -\tilde k +\sqrt{\,2\,\omega}\,e^{-\tau}\tilde k^2(\Psi_{\,\theta\theta}+\Psi\,).
\end{equation*}
 The following steps for the ACEF, see \cite{cz}, are applicable to the normalized flow \eqref{E-05}:

1) The entropy for the normalized flow,
${\cal E}(\tilde\gamma(\cdot,\tau))=\frac1{2\pi}\int_0^{2\pi}\log\tilde k\,d\theta$,
is uniformly bounded for $\tau\in[\,0,\infty)$, see \cite[pages 63--68]{cz}.
The bound on the entropy yields upper bounds for the diameter and length of the normalized flow, and also that $\tilde k$
and its gradient are uniformly bounded.

2) $e^{-\tau}\tilde k_{\rm max}(\tau)\to0$ as $\tau\to\infty$, see \cite[Lemma~3.15]{cz}.

3) The normalized curvature $\tilde k$ has a positive lower bound, see \cite[pages 70--71]{cz}.

4) With two-sided bounds for $\tilde k$, the convergence of the normalized flow \eqref{E-05}, as $\tau\to\infty$, follows.
Namely,
noting that $\,e^{-\tau}\Psi\to0$ when $\tau\to\infty$,
for any sequence $\tau_j\to\infty$, we can find a subsequence $j_k$ such that solution $\widetilde S(\cdot, \tau_{j_k})$ of \eqref{E-18a} converges in $C^\infty$ topology (as $k\to\infty$) to a solution of the corresponding stationary equation
\begin{equation}\label{E-19}
 \widetilde S = \tilde k.
\end{equation}
Based on the fact \cite{al} that the only embedded solution of \eqref{E-19} is the unit circle,
we conclude (similarly as in \cite[p.~73]{cz} for ACEF with $\Phi=1$) that $\tilde\gamma(\cdot,\,\tau)$ converges, as $\tau\to\infty$, to the unit circle in $C^\infty$,
that completes the proof of Theorem~\ref{T-02main}.

\section{Conclusion}

The main contribution of this paper is a geometrical proof of convergence of CSF for convex closed curves in a metric-affine plane.
In the future, we will study several related problems on convergence of flows in metric-affine geometry, for example:

1) CSF \eqref{E-05} for non-constant contorsion tensor $\mathfrak{T}$ and for not just convex $\gamma_0$,

2) Anisotropic CSF in metric-affine geometry,

3) The mean curvature flow for convex hypersurfaces in the metric-affine $\RR^n$.

4) Numerical experiments (as in \cite{al}) for solutions of \eqref{E-05} when $\gamma_0$ is not embedded.

\bigskip

\textbf{Acknowledgments}.
The author would like to thank P.~Walczak (University of Lodz) for helpful comments concerning the manuscript.


\begin{thebibliography}{XX}

\bibitem{al}
U. Abresch and J. Langer, The normalized curve shortening flow and homothetic solutions, J. Diff. Geom. 23 (1986), no. 2, 175--196.

\bibitem{cl}
E. Calvi\~{n}o-Louzao, et al. \textit{Aspects of differential geometry, IV}.
Synthesis Lectures on Mathematics and Statistics 26, Morgan and Claypool Publishers, 2019.

\bibitem{cz}
K.-S. Chou and Xi-P. Zhu, \textit{The curve shortening problem}. Chapman and Hall/CRC, Boca Raton, FL, 2001.

\bibitem{GH86}
M. Gage and R.S. Hamilton, The heat equation shrinking convex plane curves, J. Differential Geometry, 23\,:\,1 (1986), 69--96.

\bibitem{Gr}
M.A. Grayson, The heat equation shrinks embedded plane curves to round points, J. Differ. Geom. 26 (1987), 285--314.

\bibitem{hww}
S. He, G. Wheeler and V.M. Wheeler,
On a curvature flow model for embryonic epidermal wound healing, Nonlinear Anal. 189 (2019), 111581, 41 pp.

\bibitem{Hi}
M. Hiroshi, Differential geometrical foundations of information geometry. Geometry of statistical manifolds and divergences.
World Scientific, 2020.

\bibitem{mikes}
J. Mike\v{s}, et al. \textit{Differential geometry of special mappings}, 2nd edition. Palack\'{y} Univercity, Olomouc, 2019.

\bibitem{Yano}
K. Yano, On semi-symmetric metric connection. Rev. Roumaine, Math. Pures Appl. 15 (1970), 1579--1586.

\bibitem{zhu}
Xi-P. Zhu, \textit{Lectures on Mean Curvature Flow}. AMS/IP Stud. Adv. Math., 32, AMS, Providence, RI, 2002.

\end{thebibliography}
\end{document}